\documentclass{article}

\usepackage{graphicx} 
\usepackage{amsthm}
\usepackage{amsmath}
\usepackage{amssymb,  mathrsfs}
\usepackage{enumitem}
\usepackage{tikz}
\theoremstyle{plain}
\newtheorem{theorem}{Theorem}[section] 
\newtheorem{corollary}[theorem]{Corollary}
\newtheorem{lemma}[theorem]{Lemma}
\newtheorem{proposition}[theorem]{Proposition}
 
\newtheorem{remark}[theorem]{Remark} 
\newtheorem{example}[theorem]{Example} 
\newtheorem*{corollary*}{Corollary}

\usepackage{tikz-cd}

\theoremstyle{definition}
\newtheorem{definition}[theorem]{Definition}
\newtheorem{notation}{Notation}
\newtheorem{problem}{Problem}

\numberwithin{equation}{section}

\DeclareMathOperator{\Aut}{Aut} 
\DeclareMathOperator{\suc}{Suc}

\newcommand{\udsubseteq}{\mathrel{
\rotatebox[origin=c]{45}{$\subsetneq$}}
}
\newcommand{\ddsubseteq}{\mathrel{
\rotatebox[origin=c]{-45}{$\subsetneq$}}
}

\newcommand{\htt}{\mbox{$\text{ht}$}}
\newcommand{\cyc}{\mbox{$\text{cyc}$}}

\providecommand{\keywords}[1]
{
  \small	
  \textbf{\textit{Keywords:}} #1
}

\providecommand{\subjectclass}[1]
{
  \small	
  \textbf{\textit{Subject Classification 2020:}} #1
}

\begin{document}

\markboth{Thomas Quinn-Gregson}
{Ultrahomogeneity of monounary algebras}

\title{Ultrahomogeneity of monounary algebras%
} 
 
\author{Thomas Quinn-Gregson
} 
\maketitle
 
\begin{abstract}  A structure is \textit{ultrahomogeneous} if every isomorphism between finitely generated substructures can be extended to an automorphism of the   structure. 
The aim of this paper is to provide an explicit description of ultrahomogeneous monounary algebras of arbitrary cardinalities. We also show that a monounary algebra  is  ultrahomogeneous  if and only if it is 1-ultrahomogeneous.  
These results are then viewed in the context of a several  related concepts, including $\omega$-categoricity,  homomorphism-homogeneity,  (partial-)homogeneity, and transitivity. 
\end{abstract}


\subjectclass{08A60, 03C35, 03C10}
\maketitle


\section{Introduction}
A monounary algebra is a set together with a single unary function.
Their transparent structure has allowed for a complete description of those monounary algebras which satisfy certain  properties which arise in model theory. 
Examples include $\omega$-categoricity \cite{catmono}, homomorphism-homogeneity \cite{monohomhom},    and polymorphism-homogeneity \cite{polyhom}; these are conditions on  the automorphism  group, the endomorphism monoid, and the polymorphism clone of a monounary algebra, respectively. 
Such a complete picture   is rare amongst structures, and indeed no such classifications exist even for classical structures such as  graphs, groups, or partially ordered sets. 
Spurred by this, we consider one of the most widely studied model-theoretic properties of this form; ultrahomogeneity. 

A (first order) structure $\mathcal{M}$ is called $n$-\textit{ultrahomogeneous} if every isomorphism between $n$-generated substructures of $\mathcal{M}$ extends to an automorphism of $\mathcal{M}$, and is \textit{ultrahomogeneous}\footnote{This is also known as being \textit{homogeneous}.} if it is $n$-ultrahomogeneous for each $n\in \mathbb{N}$. 
For several key classes of relational structures, ultrahomogeneity proves to be a strong enough condition to allow for classifications, albeit often only in the countable case. These include  graphs \cite{Lachlan80}, digraphs \cite{CherlinDigraph}, and  partially ordered sets  \cite{Schmerl79}. 
 For  algebras the problem appears far more difficult in general, and indeed Saracino and Wood \cite{SaracinoWood} showed the existence of $2^\omega$ ultrahomogeneous countable groups. 
  Despite this, ultrahomogeneous abelian groups (of arbitrary cardinality) have been described in \cite{Cherlin91}, and ultrahomogeneous finite groups in \cite{Cherlin2000}. Progress has also been made  for rings and non-abelian groups  (see, e.g., \cite{Cherlin91,Cherlin93}) and for semilattices and semigroups (see, e.g., \cite{Truss99,Quinnband,Quinninv}).

The main result  of this article is a classification of the ultrahomogeneous  monounary algebras of arbitrary cardinality. 
This answers in far greater generality the first half of Problem 1 in \cite{Sokic}, which asks for all countable  locally finite ultrahomogeneous monounary algebras (here stated in the language of Fra\"iss\'e limits). In fact, this problem was previously implicitly answered by A. Ivanov in \cite{QEmono}, where he studied the condition of being \textit{strongly ultrahomogeneous (SU)}, motivated by the related property of \textit{quantifier elimination}. A structure $\mathcal{M}$  is SU  if every isomorphism between substructures of $\mathcal{M}$ generated by sets of cardinality less than $|M|$   can be extended to an automorphism of $\mathcal{M}$. Clearly for countable structures being SU is equivalent to being ultrahomogeneous, and countable SU monounary algebras are briefly  described in \cite[Theorem 2.2]{QEmono}. 

 We will also show that for a monounary algebra 1-ultrahomogeneity is sufficient for  ultrahomogeneity. This ``pointwise sufficiency" is then shown to hold for  both   homomorphism-homogeneity and   $\omega$-categoricity. 

These results are then used to provide several other classifications of  similar conditions relating to  the automorphism group of a monounary algebra. Namely, we describe all  (partial)  homogeneous monounary algebras, that is, ones in which isomorphisms between finite (partial)  subalgebras  extend to an automorphisms. 
Moreover, we show the following pleasing connection in greater generality in Theorem \ref{thm: lattice of conditions}:

\begin{corollary*}
    For the class of monounary algebras we have 
  \begin{align*}
       \text{partially homogeneous}   &\Rightarrow  \text{ultrahomogeneous}  \Rightarrow \text{homogeneous}. 
   \end{align*}
\end{corollary*}

This paper is ordered as follows. In Sections 2 and 3 we give a self-contained introduction to model theory and monounary algebras, respectively. 
 In Section 4  we produce an explicit classification of ultrahomogeneous monounary algebras of arbitrary cardinalities. 
 In Section 5
we consider the interplay between ultrahomogeneity  and similar conditions on the automorphism group of a monounary algebra, including  $n$-homogeneity, $n$-partial homogeneity,  ultrahomogeneity of the relational form, and transitivity.  In particular, we provide classifications of homogeneous  monounary algebra, partially homogeneous  monounary algebra,  and ultrahomogeneous directed pseudoforests. 
In Section 6 we discuss the limits of generalizing our results to unary algebras. 

Throughout, we let $\mathbb{N}=\{1,2,\dots\}$ and  $\omega=\mathbb{N}\cup \{0\}$. 

\section{Model theory preliminaries} 

In this section we give an introduction to the required model-theoretic notions, and  refer the reader to \cite{Hodges97} for a detailed study.

A \textit{signature} $\tau$  is a set of relation symbols and function symbols,  where  each  symbol  is  associated  with  a  natural  number,  called  its \textit{arity}.
 A  \textit{$\tau$-structure} $\mathcal{M}$ is a tuple $(M; (X^{\mathcal{M}})_{X\in \tau})$ where $M$ is a set (call the \textit{domain}) and $X^{\mathcal{M}}\subseteq M^n$ if $X$ is a relation symbol of arity $n$, and $X^{\mathcal{M}}\colon M^n\rightarrow M$ if $X$ is a function of arity $n$; in the latter case we allow $n=0$ to model constant symbols. 
When no confusion may arise  we use the same symbol for a relation and its relation symbol, and likewise for functions.
 Unless stated otherwise, we will assume that the   structures $\mathcal{A,B,C},\dots$ have domains $A,B,C,\dots$. 

\begin{definition}\label{defn:hom} Let  $\mathcal{M}$ and $\mathcal{N}$ be a pair of   $\tau$-structures.  A bijection $g\colon M\rightarrow N$ is called an  \textit{isomorphism} if 
 \begin{itemize}[leftmargin=*]
\item  $(x_1,\dots,x_n)\in R^{\mathcal{M}} \Leftrightarrow (g(x_1),\dots,g(x_n))\in R^{\mathcal{N}}$  for  every  relation symbol $R\in \tau$, and 
\item  $g(f^{\mathcal{M}}(x_1,\dots,x_n))= f^{\mathcal{N}}(g(x_1),\dots,g(x_n))$ for every function symbol $f\in \tau$. 
\end{itemize}
If, further, $\mathcal{M}=\mathcal{N}$ then $g$ is called an \textit{automorphism}.
   The group  of  automorphisms of $\mathcal{M}$ is denoted by $\Aut(\mathcal{M})$. 
  \end{definition}

A \textit{(binary) relational structure} is one in which its signature has only (binary) relation symbols.

\subsection{Permutation groups and orbit growth} 

Let $G$ be a permutation group on a set $X$, that is,  a subgroup of the symmetric group Sym$(X)$. 
  An \textit{orbit} of $G$ is a set of the form $\{g(x):g\in G\}$ for some $x\in X$. For each $n\in \mathbb{N}$ there is a natural action of $G$ on $X^n$, where $g (x_1,\dots,x_n) = (g(x_1),\dots,g(x_n))$; we define an \textit{$n$-orbit} of $G$ to be a subset of $G^n$ equal to an orbit of this action.  We let $o_n(G)$ denotes the number of $n$-orbits of $G$. 

\begin{definition} Let $G$ be a permutation group on a set $X$. 
\begin{itemize}[leftmargin=*]
\item If $o_n(G)$ is finite for each $n\in \mathbb{N}$ then we call $G$ \textit{oligomorphic}.
\item If  $o_1(G)=1$,  that is, if for every $x,y\in X$ there exists $g\in G$ with $g(x)=y$,  then we call $G$ \textit{transitive}. 
\end{itemize}  
\end{definition} 

Given a structure $\mathcal{M}$, we let
 $o_n(\mathcal{M}):= o_n(\Aut(\mathcal{M}))$, and by an \textit{$n$-orbit of $\mathcal{M}$} we will mean an $n$-orbit of the action of $\Aut(\mathcal{M})$ on $M$. Similarly, we call $\mathcal{M}$ \textit{transitive} if $\Aut(\mathcal{M})$ is. 

\subsection{Ultrahomogeneity} 

\begin{definition} A   structure  $\mathcal{M}$  is called $n$-\textit{ultrahomogeneous} if every isomorphism between $n$-generated substructures of $\mathcal{M}$ extends to an automorphism of $\mathcal{M}$, and is \textit{ultrahomogeneous} if it is $n$-ultrahomogeneous for each $n\in \mathbb{N}$. 
\end{definition} 


 An alternative description of countable ultrahomogeneous structures was given by Fra\"iss\'e \cite{Fraisse}, originally only for linear orders and then extended to all countable signatures.
  Let $\mathfrak{C}$ be a class of finitely generated $\tau$-structures for some signature $\tau$. We say that $\mathfrak{C}$  has the
 \begin{itemize}[leftmargin=*]
 \item   \textit{hereditary property (HP)} if it is closed under substructures; 
 \item \textit{joint embedding property (JEP)} if every pair $\mathcal{A,B}\in \mathfrak{C}$ embeds into some $\mathcal{C}\in \mathfrak{C}$;  
 \item   \textit{amalgamation property (AP)} if whenever $\mathcal{A},\mathcal{B}_1,\mathcal{B}_2\in \mathfrak{C}$ and $f_i\colon \mathcal{A}\rightarrow \mathcal{B}_i$ is an embedding $(i=1,2)$, then there exists $\mathcal{C}\in \mathfrak{C}$ and embeddings $g_i\colon \mathcal{B}_i\rightarrow \mathcal{C}$ such that $g_1\circ f_1(a) = g_2\circ f_2(a)$ for each $a\in A$. 
 \end{itemize}
 
The \textit{age} of a structure $\mathcal{M}$,  denoted age$(\mathcal{M})$, is the class of all finitely generated structures which embed in $\mathcal{M}$. 

A structure is called  \textit{locally finite} if  every finitely generated substructure is finite.
 
 \begin{theorem}[Fra\"iss\'e's Theorem] Let $\tau$ be a countable signature and let $\mathfrak{C}$ be a non-empty class of finitely generated (finite) $\tau$-structures which is countable, up to isomorphism, and has the HP, JEP, and AP. 
  Then there is a unique countable (locally finite) $\tau$-structure $\mathcal{M}$ which is ultrahomogeneous and with age$(\mathcal{M})$= $\mathfrak{C}$; we call $\mathcal{M}$ the \textit{Fra\"iss\'e limit} of $\mathfrak{C}$. 
 \end{theorem}

 \section{Monounary algebras preliminaries} 
 
A  \textit{partial monounary algebra}  is a pair $\mathcal{A}=(A; \theta)$, where $A$ is a nonempty set and $\theta$ a partial unary operation on $A$; we let dom$(\theta)$ denote the domain of $\theta$. 
  If   dom$(\theta)$=$A$, then   $\mathcal{A}$ is a \textit{monounary algebra}. More generally, a \textit{unary algebra} is a set together with a finite collection of unary operations. 
  
Each non-empty subset $B$ of a monounary algebra $\mathcal{A} $ forms a partial monounary algebra $\mathcal{B}=(B;\theta_B)$, where dom$(\theta_B)$=$\{b\in B:\theta(b)\in B\}$, and $\theta_B(b)=\theta(b)$ if $b\in$ dom$(\theta_B)$; we call $\mathcal{B}$ a \textit{partial subalgebra} of $\mathcal{A}$. If,  further,  $\mathcal{B}$ forms a monounary algebra then  it is called a \textit{subalgebra} of $\mathcal{A}$. We call $\mathcal{A}$ \textit{connected} if for every $x,y\in A$ there exists $n,m\in \omega$ with $\theta^n(x)=\theta^m(y)$. 
 A maximal connected subalgebra of $\mathcal{A}$ is called a \textit{connected component}, and the set of connected components of $\mathcal{A}$ is denoted by $\mathfrak{C}(\mathcal{A})$. 
 
\begin{notation} \label{not:sum} Given a collection $\{\mathcal{A}_i:i\in I\}$  of pairwise disjoint monounary algebras, we let $\sum_{i\in I} \mathcal{A}_i$ denote the monounary algebra formed by taking the disjoint union of the $\mathcal{A}_i$.
 If there exists a monounary algebra $\mathcal{A}$ with $\mathcal{A}_i\cong \mathcal{A}$ for every $i\in I$  then we may simply denote $\sum_{i\in I}\mathcal{A}_i$ as $I \cdot \mathcal{A}$. 
\end{notation} 
Hence, for any monounary algebra $\mathcal{A}$ with $\mathfrak{C}(\mathcal{A})=\{\mathcal{A}_i:i\in I\}$ we have $\mathcal{A}=\sum_{i\in I} \mathcal{A}_i$. 

A bijection $\phi:\mathcal{A}\rightarrow \mathcal{B}$ between a pair of partial monounary algebras $\mathcal{A}=(A;\theta)$ and $\mathcal{B}=(B;\theta')$ is an \textit{isomorphism} if, for all $x\in A$, we have $x\in \text{dom}(\theta)$ if and only if  $\phi(x)\in \text{dom}(\theta')$,    and in which case   $\phi(\theta(x))=\theta'(\phi(x)))$. 
If $\mathcal{A}$ and $\mathcal{B}$ are both monounary algebras (so dom$(\theta)$=$A$ and dom$(\theta')$=$B)$ then the condition simplifies to $\phi(\theta(x))=\theta'(\phi(x))$ for each $x\in A$, and we recover Definition \ref{defn:hom}.

For the remainder of the subsection we let $\mathcal{A}=(A;\theta)$ be a monounary algebra. 
Given a  non-empty subset $B$ of ${A}$, we let $\langle B \rangle$ denote the subalgebra of $\mathcal{A}$ generated by $B$, so that $\langle B \rangle = \{\theta^n(b): b\in B, n\in \omega\}$.

An element $x\in A$ is \textit{cyclic} if $\theta^n(x)=x$ for some $n\in \mathbb{N}$, and is called \textit{acyclic} otherwise. 
The set of all cyclic elements of $\mathcal{A}$ is denoted  $\cyc(\mathcal{A})$ and, if  $\cyc(A)\neq \emptyset$, then it forms a subalgebra of $\mathcal{A}$. 
A \textit{cyclic algebra} is a connected monounary algebra in which every element is cyclic.
For each $n\in \mathbb{N}$, there exists a cyclic algebra of size $n$ (also called an \textit{$n$-cycle}),  which is unique up to isomorphism  and denoted  by $\mathcal{Z}_n=(\{0,1,\dots,n-1\};\suc)$, where $\suc(i)=i+1 \mod n$.  

The \textit{height} of an element $x\in A$, denoted $\htt(x)$, is the least $n\in \omega$ such that $\theta^n(x)$ is cyclic; if no such $n$ exists then we put $\htt(x)=\infty$.
The \textit{height} of $\mathcal{A}$, denoted $\htt(\mathcal{A})$, is the least $n\in \omega$ such that  $\htt(x) \leq n$ for every $x\in A$; if no such $n$ exists then we write  $\htt(\mathcal{A})=\infty$. 

If $x\in \mathcal{B}\in \mathfrak{C}(\mathcal{A})$ has finite height, then so does every element of $\mathcal{B}$. 
It follows that  $\cyc(\mathcal{B})\neq \emptyset$ if and only if some  element of $B$  has finite height. 

\begin{lemma}\label{lemma: subalgebra} Let $\mathcal{B}$ be a subalgebra of a monounary algebra $\mathcal{A}=(A;\theta)$. 
\begin{enumerate}
\item[$(\mathrm{i})$] If $\mathcal{C},\mathcal{C}'\in \mathfrak{C}(\mathcal{B})$ are distinct, then there exist distinct $\mathcal{D},\mathcal{D}'\in \mathfrak{C}(\mathcal{A})$ with $C\subseteq D$ and $C' \subseteq D'$. 
\item[$(\mathrm{ii})$] For any $\mathcal{C}\in \mathfrak{C}(\mathcal{A})$, if $B\cap C\neq \emptyset$ then  $\cyc(\mathcal{C})\subseteq \mathcal{B}$.   
\end{enumerate}
\end{lemma}

\begin{proof} (i) For each $x\in B$ and each $n\in \omega$  we have $\theta^n(x)\in B$  since $\mathcal{B}$ is a subalgebra of $\mathcal{A}$.
 Consequently, elements $x$ and $y$ of $\mathcal{B}$ are connected in $\mathcal{A}$ if and only if they are connected in $\mathcal{B}$, to which the result follows. 

(ii) Let $b\in B\cap C$. If $\htt(b)=\infty$ then  $\cyc(\mathcal{C})$ is empty. Otherwise, there exists $n\in \omega$  such that $\theta^n(b)=d$ is cyclic.
 Then for each $c\in \cyc(\mathcal{C})$  there exists $m\in \omega$ with $\theta^m(d)=c$, so $\theta^{n+m}(b)=c \in B$. 
\end{proof}

We call a monounary algebra \textit{$n$-generated} if it can be generated by $n$  elements. 
Every infinite 1-generated monounary algebra is isomorphic to $\mathcal{N} :=(\mathbb{N};\suc)$ with $\suc(n)=n+1$.   
Notice that  $\cyc(\mathcal{N})$ is empty. 

\begin{lemma} \label{lemma:cyclic stuff} Let $\mathcal{A}=(A;\theta)$ be a monounary algebra and $\mathcal{B},\mathcal{B}'\in \mathfrak{C}(\mathcal{A})$. 
Then for any $x\in B$ and $y\in B'$ the following are equivalent: 
\begin{enumerate} 
\item[$(1)$]  $\langle x \rangle \cong \langle y \rangle$;
\item[$(2)$]  $\htt(x) = \htt(y)$ and, if  $\htt(x)\neq \infty$, then   $\cyc(\mathcal{B}) \cong \cyc(\mathcal{B}')$; 
\item[$(3)$]  $\htt(x)=\htt(y)$ and $|\cyc(\mathcal{B})|=|\cyc(\mathcal{B}')|$. 
\end{enumerate}
\end{lemma} 

\begin{proof}
If  $\htt(x) = \infty$ then  $\cyc(\mathcal{B})$ is empty and $\langle x \rangle$ is isomorphic to $\mathcal{N}$, to which the result follows. 

Suppose instead that  $\htt(x) = n$ is finite. Then $\theta^n(x)\in \cyc(\mathcal{B})$, and so  $\cyc(\mathcal{B})\subseteq \langle x \rangle$ by Lemma \ref{lemma: subalgebra} (ii). 
Hence $\langle x \rangle = \{x,\theta(x),\dots \theta^{n-1}(x)\}\cup \cyc(\mathcal{B})$, and similarly if  $\htt(y) = m$ then $\langle y \rangle = \{y,\theta(y),\dots,\theta^{m-1}(y)\}\cup \cyc(\mathcal{B}')$.
The equivalence of (1) and (2) is then clear. 
Cyclic subalgebras of the same size are isomorphic, so the equivalence (2) and (3) holds.  
\end{proof}

It follows that, for each $n\in \omega$ and $k\in \mathbb{N}$, there exists a unique, up to isomorphism, 1-generated monounary algebra $\langle x \rangle$ with  $\htt(x)=n$ and  $\cyc(\langle x \rangle)\cong \mathcal{Z}_k$. 

\subsection{The relational form} 

\begin{definition} Given a partial monounary algebra $\mathcal{A}=(A;\theta)$, its \textit{relational form} is the digraph  $\mathcal{A}^*=(A;R_{\theta})$  where $(x,y)\in R_{\theta}$ if and only if $\theta (x) =y$.  
\end{definition} 


Note that $\mathcal{A}^*$  is loopless if and only if $\mathcal{A}$ does not contain a 1-cycle. Hence in this case $\mathcal{A}^*$ forms a \textit{directed pseudoforest} (also known as a \textit{functional graph}), that is, a loopless digraph in which  each vertex has at most one outgoing edge (see \cite[Section 3.4]{Harary} for a detailed study into directed pseudoforests). Conversely, each directed pseudoforest is the relational form of some partial monounary algebra.  

\subsection{Isomorphisms between monounary algebras} 

  The results in this subsection will be vital in the subsequent section,  in particular when extending certain partial maps. The first is a simple extension of \cite[Lemma 1.10]{2-hom}, and shows isomorphisms between monounary algebras must preserve connectivity: 

\begin{lemma}\label{lemma: iso con} Let $\phi\colon \mathcal{A}\rightarrow \mathcal{B}$ be an isomorphism between a pair of monounary algebras $\mathcal{A}$ and $\mathcal{B}$.
 Let $\mathcal{C}\in \mathfrak{C}(\mathcal{A})$ and $\mathcal{D}\in \mathfrak{C}(\mathcal{B})$ be such that there exists $x\in C$ with $\phi(x)\in D$. Then $\phi(\mathcal{C})=\mathcal{D}$. 
\end{lemma}

As a consequence, isomorphisms between monounary algebras are simply built from isomorphisms between their connected components: 

\begin{corollary} \label{cor: iso con} Let $\mathcal{A}$ and $\mathcal{B}$ be a pair of monounary algebras.
 If $\pi\colon \mathfrak{C}(\mathcal{A})\rightarrow \mathfrak{C}(\mathcal{B})$ is a  bijection  and $\phi_\mathcal{C}\colon \mathcal{C}\rightarrow \pi(\mathcal{C})$ is an isomorphism   for each $\mathcal{C}\in \mathfrak{C}(\mathcal{A})$, then $\phi=\bigcup_{\mathcal{C}\in \mathfrak{C}(\mathcal{A})} \phi_\mathcal{C}$ is an  isomorphism  from $\mathcal{A}$ to $\mathcal{B}$. Moreover, every  isomorphism  from $\mathcal{A}$ to $\mathcal{B}$ can be constructed in this way. 
\end{corollary} 

\begin{notation}\label{not:above set} Given a monounary algebra $\mathcal{A}=(A;\theta)$ and $z\in A$ we let 
\[ A_z := \{z\}\cup \{x\in A\setminus \cyc(\mathcal{A}):(\exists n\in \mathbb{N})\, \theta^n(x)=z, \theta^{n-1}(x)\notin \cyc(\mathcal{A})\}. 
\] 
\end{notation} 
Then each $\mathcal{A}_z=(A_z;\theta)$ is a partial subalgebra, and if $z$ is acyclic then 
\[ A_z=\bigcup_{k\in \omega} \theta^{-k}(z).\]
Note that $(A\setminus A_z)\cup \{z\}$ is a subalgebra of $\mathcal{A}$. Moreover, if $\mathcal{A}$ is connected with  $\cyc(\mathcal{A})\neq \emptyset$ then   $\{A_c: c\in \cyc(\mathcal{A})\}$ is a partition of $A$. 

The following lemma allows for isomorphisms between the partial subalgebras $\mathcal{A}_z$  to be inductively built `from above'. The proof is immediate from definitions, and as such we skip it. 

\begin{lemma} \label{lemma:inductive built partials} Let $\mathcal{A}$ be a monounary algebra and  fix some $x,y\in A$. 
Let  $\pi\colon \theta^{-1}(x) \rightarrow \theta^{-1}(y)$ be a bijection and, for each $u\in\theta^{-1}(x)$, let $\phi_u\colon \mathcal{A}_u\rightarrow \mathcal{A}_{\pi(u)}$ be an isomorphism. 
Then the map $\Phi\colon \mathcal{A}_x\rightarrow \mathcal{A}_y$ which maps $x$ to $y$ and extends each $\phi_u$ is an isomorphism, and every isomorphism from $\mathcal{A}_x$ to  $\mathcal{A}_y$ can be constructed in this way. 
\end{lemma}

\begin{lemma}\label{lemma:cycles aut} Let $\mathcal{A}=(A;\theta)$ be a connected monounary algebra with  $\mathcal{C}=\cyc(\mathcal{A})\neq \emptyset$.
 Let $\pi$ be an automorphism of $\mathcal{C}$, and  let $\phi_c:\mathcal{A}_c\rightarrow \mathcal{A}_{\pi(c)}$  be a partial monounary algebra isomorphism for each $c\in \mathcal{C}$. 
 Then $\phi=\bigcup_{c\in \mathcal{C}} \phi_c$ is an automorphism of $\mathcal{A}$, and conversely every automorphism of $\mathcal{A}$ can be constructed in this way. 
\end{lemma} 

\begin{proof} 
Let $\pi$, $\phi_c$ $(c\in \mathcal{C})$, and $\phi$ be constructed as in our hypothesis.
 Clearly $\phi$ is a bijection, so it suffices to  prove that it forms a homomorphism. 
 Let $x\in A$, say $x\in {A}_c$.
  If $\theta(x)\in {A}_c$ then as $\phi_c$ is an isomorphism we have that $\phi_c(x)$ and $\theta(\phi_c(x))$ are elements of  ${A}_{\pi(c)}$  and 
\[ \phi(\theta(x))=\phi_c(\theta(x))=\theta(\phi_c(x))=\theta(\phi(x)).
\] 
If $\theta(x)\notin A_c$ then $x=c$ and so $\theta(x)$ is cyclic. Hence as $\pi$ is an automorphism of $\mathcal{C}$  we have 
\[ \phi(\theta(x))=\phi_{\theta(x)}(\theta(x)) = \pi(\theta(x))=\theta(\pi(x)) = \theta(\phi_x(x))=\theta(\phi(x)).
 \] 
Conversely, if $\psi$ is an automorphism of $\mathcal{A}$, then $\pi'=\psi|_{\mathcal{C}}$  is  an automorphism of $\mathcal{C}$, and $\psi|_{\mathcal{A}_c}$ is an isomorphism from $\mathcal{A}_c$ to $\mathcal{A}_{\pi'(c)}$ for each $c\in \mathcal{C}$.
\end{proof}

Notice that for any $z\in A$ the digraph $\mathcal{A}^*_z$ is a tree. 
Automorphisms of tree-like structures have been well-studied, often from the viewpoint of ultrahomogeneity (see, e.g., \cite{Droste,Hamann}).
 The following proposition is   motivated by a method used by Droste in \cite[Theorem 6.16]{Droste} when considering tree-like partially ordered sets. 
 
\begin{proposition}\label{prop:iso cut} Let $\mathcal{A}=(A;\theta)$ be a  monounary algebra and fix some $x,y\in A$.
 Suppose there exists an isomorphism $\psi$  from the subalgebra $(A\setminus A_x)\cup \{x\}$ to $(A\setminus A_y)\cup \{y\}$ which maps $x$ to $y$. 
 Then for any isomorphism $\phi\colon\mathcal{A}_x\rightarrow \mathcal{A}_y$, the map $\Psi\colon A\rightarrow A$ given by 
\[   
\Psi(a) = 
     \begin{cases}
      \phi(a) & \quad \text{if }  a\in  A_x,\\
       \psi(a) &\quad\text{otherwise } \\
     \end{cases}
\]
is an automorphism of $\mathcal{A}$. 
\end{proposition} 

\begin{proof} Since $x$ is  the unique element of $\mathcal{A}_x$ which is either a 1-cycle or has no image under $\theta$, it follows that the isomorphism $\phi$  must map $x$ to $y$. Hence $\Psi$ is a well-defined bijection.

 If $a\in A_x$ and $\theta(a)\notin A_x$ then $a=x$ and we have 
\[ \Psi(\theta(x)) = \psi(\theta(x)) = \theta(\psi(x))=\theta(\Psi(x)).
\] 
If $a,\theta(a)\in A_x$ or $a,\theta(a)\in A\setminus A_x$ then the homomorphism property is immediate from that of $\phi$ and $\psi$, respectively. Hence $\Psi$ is an automorphism of $\mathcal{A}$.
\end{proof} 
 
 We denote the map $\Psi$ by $\phi \sqcup_x \psi$.  
 
\begin{corollary}\label{cor:heightauto}  Let $\mathcal{A}=(A;\theta)$ be a   monounary algebra.
Let $x,y\in A$ be of height $k\in \omega$ and with $\mathcal{A}_{\theta^{m}(x)}\cong \mathcal{A}_{\theta^{m}(y)}$ for each $0\leq m \leq k$. 
Then there exists an automorphism of $\mathcal{A}$ which maps $x$ to $y$ if and only if  there exists an automorphism of $\mathcal{A}$ mapping  $\theta^k(x)$ to $\theta^k(y)$. 
\end{corollary} 

\begin{proof} 
 $(\Rightarrow)$ Since any automorphism of $\mathcal{A}$ which maps $x$ to $y$ preserves $\theta$, it must also map $\theta^k(x)$ to $\theta^k(y)$. 

$(\Leftarrow)$ The result is trivial if $x$ and $y$ have height 0. We may proceed by induction by supposing that the result is true for pairs of elements of height at most $k-1$. 
 Let $x,y\in A$ be of height $k$, say $\theta^k(x)=c$ and $\theta^k(y)=d$, and satisfy the hypothesis of the corollary. 
 Suppose there exists an automorphism mapping $\theta^{k}(x)$ to $\theta^k(y)$. 
 Then $\theta(x)$ and $\theta(y)$ are of height $k-1$ and also satisfies the hypothesis of the corollary,  so in particular 
 $\mathcal{A}_{\theta(x)}\cong \mathcal{A}_{\theta(y)}$. 
  Hence by the inductive hypothesis there exists an automorphism $\psi$ of $\mathcal{A}$ mapping $\theta(x)$ to $\theta(y)$. 
 By Lemma \ref{lemma:inductive built partials}  every isomorphism from $\mathcal{A}_{\theta(x)}$ to  $\mathcal{A}_{\theta(y)}$    is built from a bijection $\pi\colon \theta^{-1}(\theta(x)) \rightarrow \theta^{-1}(\theta(y))$  such that $\mathcal{A}_a\cong \mathcal{A}_{\pi(a)}$ for each $a\in \theta^{-1}(\theta(x))$. 
Since $\mathcal{A}_x \cong \mathcal{A}_y$ we may   take $\pi$ such that $\pi(x)=y$.  
 Let $\phi_a\colon\mathcal{A}_a\rightarrow \mathcal{A}_{\pi(a)}$ be an isomorphism for each $a\in \theta^{-1}(\theta(x))$, and let $\phi\colon \mathcal{A}_{\theta(x)}\rightarrow \mathcal{A}_{\theta(y)}$ be the map in which $\phi(\theta(x))=\theta(y)$ and $\phi|_{\mathcal{A}_a}=\phi_a$ for each $a\in \theta^{-1}(\theta(x))$.
  Then   $\phi$ is an isomorphism by Lemma \ref{lemma:inductive built partials}, and hence $\Psi=\phi \sqcup_{\theta(x)} \psi|_{B}$ is an automorphism of $\mathcal{A}$ by Proposition \ref{prop:iso cut}, where $B=(A\setminus A_{\theta(x)})\cup \{\theta(x)\}$.
   Moreover, $\Psi$ extends $\pi$ by construction  and hence maps $x$ to $y$. This completes the inductive step. 
\end{proof}
\begin{corollary}\label{cor:orbit induct} Let $\mathcal{A}$ be a   monounary algebra and let $\phi\colon \langle x_1,\dots,x_n\rangle \rightarrow \langle y_1,\dots,y_n \rangle$ be an isomorphism between finitely generated substructures of $\mathcal{A}$ defined by $\phi(x_i)=y_i$ for each $1\leq i \leq n$.
Suppose also that $(x_1,\dots,x_{n-1})$ and $(y_1,\dots,y_{n-1})$ are in the same $(n-1)$-orbit, and   $x_n$ and $y_n$ are in the same 1-orbit of $\mathcal{A}$. 
Then $(x_1,\dots,x_n)$ and $(y_1,\dots,y_n)$ are in the same $n$-orbit of $\mathcal{A}$, that is, $\phi$ can be extended to an automorphism of $\mathcal{A}$. 
\end{corollary}  

\begin{proof}
Let $\psi,\psi'\in \Aut(\mathcal{A})$  be such that $\psi(x_i)=y_i$ for each $1\leq i \leq n-1$ and $\psi'(x_n)=y_n$. 
Let $\mathcal{B}=\langle x_1,\dots,x_{n-1}\rangle$ and $\mathcal{B}'=\langle y_1,\dots,y_{n-1}\rangle$.
Let $\mathfrak{C}(\mathcal{A})=\{\mathcal{C}_i:i\in I\}$ with $x_k\in \mathcal{C}_{i_k}$ and $y_k \in \mathcal{C}_{j_k}$. Since $\phi$ is an isomorphism, and thus preserves connectivity, the   map $\pi \colon \{i_1, \dots , i_n\} \rightarrow \{j_1, \dots , j_n\}$ defined by $\pi(i_k) = j_k$ for each $1\leq k\leq n$ is a bijection. Moreover, as $\psi, \psi'\in \Aut(\mathcal{A})$, we have 
 $\psi(\mathcal{C}_{i_k}) = \mathcal{C}_{j_k}$   for each $k<n$ and  $\psi'(\mathcal{C}_{i_n})= \mathcal{C}_{j_n}$, so that $\pi$ preserves isomorphism types of the connected components.

Suppose first that $x_n$ is in a different connected component as the elements of $B$, i.e. $\langle x_n \rangle \cap B = \emptyset$.
Letting $J=\{j \in I \colon \mathcal{C}_j \cong \mathcal{C}_{i_n}\}$, extend $\pi|_J$ to a bijection $\pi'$ of $J$. For each $j\in J$ let $\Psi_j\colon \mathcal{C}_j\rightarrow \mathcal{C}_{\pi'(j)}$ be an isomorphism, where $\Psi_j = \psi|_{\mathcal{C}_j}$ if $j\in \{i_1, \dots , i_{n-1}\}$ and $\Psi_{i_n} = \psi'|_{\mathcal{C}_{i_n}}$. The bijection of $A$ formed as the union of the $\Psi_j$ ($j\in J)$ and the $\psi|_{\mathcal{C}_{i}}$ ($i \in I \setminus J$) is a desired automorphism of $\mathcal{A}$.

Assume instead that there exists a minimum $k\in \omega$     such that $\theta^k(x_n)=z\in B$. 
Then, as $\phi$ is an isomorphism, $k$ is also minimum such that $\theta^k(y_n)=z'\in B'$, noting also that $z'=\phi(z)$. 
If $k=0$, so that $x_n\in B$,  then $\psi$ provides the required automorphism; assume instead that $k>0$ so that $x_n\notin B$   (similarly for $y_n$). 
Since $z=\theta^k(x_n)=\theta^p(x_i)$ for some $i\in \{1,\dots,n-1\}$ and $p\in\omega$, we have 
$$z' = \phi(z) = \phi(\theta^p(x_i)) = \theta^p(\phi(x_i)) = \theta^p(y_i)=\theta^p(\psi(x_i))= \psi(\theta^p(x_i))=\psi(z)$$
and similarly $z'=\psi'(z)$.
Hence,  $\mathcal{A}_z$ and $\mathcal{A}_{z'}$ are isomorphic, and by Lemma \ref{lemma:inductive built partials} any isomorphism between them is built via a bijection $\pi\colon \theta^{-1}(z)\rightarrow \theta^{-1}(z')$ and isomorphisms $\Phi_u\colon \mathcal{A}_u\rightarrow \mathcal{A}_{\pi(u)}$ for  each $u\in\theta^{-1}(z)$. 
We choose $\pi$ so that $\pi(\theta^{k-1}(x_n))= \theta^{k-1}(y_n)$,  $\pi(u)=\psi(u)$ for each $u\in B\cap \theta^{-1}(z)$, and let $\pi(u)$ be chosen arbitrarily  subject to $\mathcal{A}_u\cong \mathcal{A}_{\pi(u)}$ for all other $u\in\theta^{-1}(z)$. 
We may then take isomorphisms $\Phi_{\theta^{k-1}(x_n)}=\psi'|_{\mathcal{A}_{\theta^{k-1}(x_n)}}$,  $\Phi_u=\psi|_{\mathcal{A}_u}$ for $u\in B\cap \theta^{-1}(z)$, and fix any isomorphism $\Phi_u$ for all other $u\in\theta^{-1}(z)$. 
Letting $\Phi\colon \mathcal{A}_z\rightarrow \mathcal{A}_{z'}$ be the resulting isomorphism, we form the automorphism $\Psi=  \Phi \sqcup_{z} \psi|_{(\mathcal{A}\setminus \mathcal{A}_z)\cup\{z\}}$ of $\mathcal{A}$.
 Then by construction $\Psi|_B=\psi|_B$, and in particular $\Psi(x_i)=y_i$ for $1\leq i \leq n-1$.
 Moreover, $\Psi(x_n)=\Phi(x_n)=\Phi_{\theta^{k-1}(x_n)}(x_n) = y_n$ as required. 
\end{proof}

\section{Ultrahomogeneous monounary algebras} 

  Several alternative notions of ultrahomogeneity have been investigated for monounary algebras, and we describe several here. 
   A structure $\mathcal{M}$ is called \textit{$n$-homogeneous}     for some $n\in \mathbb{N}$ if every isomorphism between substructures of $\mathcal{M}$ of size $n$ extends to an automorphism of $\mathcal{M}$.
 For locally finite structures,  being ultrahomogeneous is equivalent to being $n$-homogeneous for each $n$.
     For arbitrary structures this needs not hold, since structures without finite substructures are trivially $n$-homogeneous for every $n$, but need not be ultrahomogeneous.
      For example, $\mathcal{N}=(\mathbb{N};\suc)$ has no finite substructures and is not ultrahomogeneous,  since the isomorphism mapping $\langle 2 \rangle=(\{2,3,\dots\};\suc)$ to $\langle 1 \rangle=\mathcal{N}$ cannot be extended to an automorphism of $\mathcal{N}$.  
     
  The   $n$-homogeneity of monounary algebras is classified in \cite{1hom} for $n=1$ and in \cite{2-hom} for $n=2$.
 This property is strengthened in  \cite{1hom}   and \cite{partial2hom}, where isomorphisms  are taken between \textit{partial} subalgebras  of some fixed size $n$ (called here ``partially $n$-homogeneous''); we discuss this in more detail in the subsequent section. 
      Transitive monounary algebras were classified in \cite{monohom} (called here ``homogeneous''). The following result merges Theorem 1.8 and Lemma 1.5 of that paper. 
      
      \begin{theorem}[\protect{\cite{monohom}}]   \label{thm:transitive}  A monounary algebra $\mathcal{A}$ is transitive if and only either 
      \begin{itemize}
      \item[$(\mathrm{i})$] $\mathcal{A}$ consists of cycles with the same cardinality,
       \item[$(\mathrm{ii})$] $\mathcal{A}$  is cycle-free  and there is a cardinal $\alpha >0$ such that $|\theta^{-1}(x)|=\alpha$   for each $x \in A$.
      \end{itemize}
      Moreover, for each cardinal $\alpha>0$, there exists a unique, up to isomorphism,  connected monounary algebra fulfilling $(\mathrm{ii})$, denoted $\mathcal{B}_\alpha$. 
      \end{theorem}
      
For a construction of $\mathcal{B}_\alpha$ we  refer to \cite[Notation 2.2]{monohom}.

\textit{Homomorphism-homogeneous} and \textit{polymorphism-homogeneous}   monounary algebras are described in \cite{monohomhom} and \cite{polyhom}, respectively. 
      While ultrahomogeneity is a property on the group of automorphisms, these are the corresponding properties on the  endomorphism monoid and polymorphism clone, respectively. 
 
Despite these breakthroughs, determining the ultrahomogeneity of a monounary algebra remains an open problem. 
The countable case was succinctly solved in \cite[Theorem 2.2]{QEmono}. Later in \cite{Sokic}, a  countable list of Fra\"iss\'e classes of finite monounary algebras was constructed;
 the class of all finite monounary algebras $\mathfrak{F}$ and, for each $k\geq 1$, the class $\mathfrak{F}_k$ of all finite monounary algebras $(A;\theta)$ in which  $|\theta^{-1}(a)|\leq k$ for each $a\in A$. 
The first half of Problem 1 in the latter paper asks for all Fra\"iss\'e subclasses of $\mathfrak{F}$ or, alternatively,  for all countable locally finite ultrahomogeneous monounary algebras. While their problem is implicitly solved in  \cite{QEmono}, we provide an explicit description  in Corollary \ref{cor:HomogLF} of all  countable monounary algebras in which their age forms a  Fra\"iss\'e class. 

In this section we describe all ultrahomogeneous monounary algebras of arbitrary cardinality.
Our first main result is the following simple consequence of Corollary \ref{cor:orbit induct}, which shows that   1-ultrahomogeneity is sufficient for ultrahomogeneity. 
This result has precedent; in \cite{monohomhom}, the proofs of Lemma 2 and the forward direction to Lemma 4 required only that any homomorphism  between 1-generated subalgebras extends to an endomorphism. 
It is then a simple exercise to show that every \textit{1-homomorphism-homogeneous} monounary algebra is homomorphism-homogeneous. 

\begin{proposition} \label{prop:early hom iff 1hom} 
A monounary algebra  is ultrahomogeneous if and only if it is 1-ultrahomogeneous. 
\end{proposition} 

\begin{proof} 
($\Rightarrow$) Immediate.   

($\Leftarrow$) 
 For some fixed $n>1$, assume that the monounary algebra $\mathcal{A}$ is $k$-ultrahomogeneous for each $k<n$.
 Consider an isomorphism $\phi\colon\mathcal{B}\rightarrow \mathcal{B}'$ between $n$-generated subalgebras of $\mathcal{A}$, say   $\mathcal{B}=\langle x_1,\dots,x_n \rangle$ and $\mathcal{B}'=\langle y_1,\dots,y_n \rangle$ with $\phi(x_i)=y_i$ for each $1\leq i \leq n$.  
  Let $\mathcal{D}=\langle x_1,\dots,x_{n-1} \rangle$, $\mathcal{D}'=\langle y_1,\dots,y_{n-1}\rangle$, and $\phi'=\phi|_D\colon \mathcal{D}\rightarrow \mathcal{D}'$. 
  Since $\mathcal{A}$ is $(n-1)$-ultrahomogeneous we may extend $\phi'$ to an automorphism $\psi$ of $\mathcal{A}$.
    Moreover, as  $\mathcal{A}$ is 1-ultrahomogeneous and $\langle x_n \rangle \cong \langle y_n \rangle$, there exists an automorphism $\varphi$ of $\mathcal{A}$ with $\varphi(x_n)=y_n$. 
    By Corollary \ref{cor:orbit induct}
    there exists an automorphism of $\mathcal{A}$ extending $\phi$, and so $\mathcal{A}$ is $n$-ultrahomogeneous. The
 result follows inductively. 
\end{proof}

If $\mathcal{A}$ is an algebra and $\mathcal{B}$ is a subalgebra of $\mathcal{A}$ with $B\neq A$ then we call $\mathcal{B}$ a \textit{proper subalgebra}. 
Recall that any algebra without proper subalgebras is trivially ultrahomogeneous. 
Hence, as a cyclic algebra   is generated by any of its elements we obtain: 
 
\begin{lemma} \label{lem:cyclichom} For each $n\in \mathbb{N}$ the cycle $\mathcal{Z}_n$ is ultrahomogeneous.  
\end{lemma} 

While the following lemma is stated in  \cite[Lemma 4.28]{Quinninv}  for inverse semigroups, the proof trivially translates to any algebra  and as such we skip it:  

 \begin{lemma} \label{lemma:homs iso}   Let $\mathcal{A}$ and $\mathcal{B}$ be a pair of isomorphic ultrahomogeneous monounary algebras.
 Then any isomorphism between finitely generated subalgebras of $\mathcal{A}$ and $\mathcal{B}$ can be extended to an isomorphism from $\mathcal{A}$ to $\mathcal{B}$. 
\end{lemma}

\begin{proposition}\label{prop:connected}  A monounary algebra $\mathcal{A}$ is  ultrahomogeneous if and only if each $\mathcal{B}\in \mathfrak{C}(\mathcal{A})$ is   ultrahomogeneous and whenever $\mathcal{B},\mathcal{B}'\in \mathfrak{C}(\mathcal{A})$ are such that $|\cyc(\mathcal{B})|=|\cyc(\mathcal{B}')|$ then $\mathcal{B}\cong \mathcal{B}'$. 
\end{proposition} 

\begin{proof}
 $(\Rightarrow)$ Let $\mathcal{B}\in \mathfrak{C}(\mathcal{A})$, and let $\phi$  an isomorphism between finitely generated   subalgebras of $\mathcal{B}$. 
 Then as $\mathcal{A}$ is  ultrahomogeneous there exists an automorphism of $\mathcal{A}$ extending $\phi$ which,  by Lemma \ref{lemma: iso con}, restricts to an automorphism of $\mathcal{B}$.
  Hence $\mathcal{B}$ is  ultrahomogeneous. 

Now suppose $\mathcal{B},\mathcal{B}'\in \mathfrak{C}(\mathcal{A})$ are such that $|\cyc(\mathcal{B})|=|\cyc(\mathcal{B}')|=k$.
 It suffices to show that $\mathcal{B}$ and $\mathcal{B}'$ possess isomorphic finitely generated  subalgebras, since any isomorphism between them can be extended to an automorphism of $\mathcal{A}$, which then restricts to an isomorphism between $\mathcal{B}$ and $\mathcal{B}'$ again by Lemma \ref{lemma: iso con}. 
 If $k=0$ then  any $x\in B, y\in B'$ are such that $\langle x \rangle$ and $\langle y \rangle$ are isomorphic (to $\mathcal{N}$). 
 If $k>0$ then  $\cyc(\mathcal{B})\cong \cyc(\mathcal{B}')$ are 1-generated.
 
$(\Leftarrow)$ By Proposition \ref{prop:early hom iff 1hom} it suffices to show that $\mathcal{A}$ is 1-ultrahomogeneous.
Let $\phi\colon \langle x \rangle\rightarrow \langle y\rangle$ be an isomorphism and let $\mathcal{B},\mathcal{B}'\in \mathfrak{C}(\mathcal{A})$ be such that $x\in  {B}$ and $y\in  {B}'$.
Since $\phi$ is an isomorphism it follows from Lemma \ref{lemma:cyclic stuff} that $|\cyc(\mathcal{B})|=|\cyc(\mathcal{B}')|$, and hence $\mathcal{B}\cong \mathcal{B}'$. 
By  the ultrahomogeneity of  $\mathcal{B}$  and $\mathcal{B}'$  we may extend $\phi$ to an isomorphism $\varphi\colon \mathcal{B}\rightarrow \mathcal{B}'$   by Lemma \ref{lemma:homs iso}.
Let $\pi$ be a bijection of $\mathfrak{C}(\mathcal{A})$ which preserves the isomorphism types of the connected components and with $\pi(\mathcal{B})=\mathcal{C}$. 
Then by Lemma \ref{cor: iso con} we may build an automorphism  of $\mathcal{A}$ via $\pi$ which extends $\varphi$, and thus extends $\phi$. 
    	\end{proof}

To classify  ultrahomogeneous monounary algebras it therefore suffices to consider the connected case. For this we require the following auxiliary lemma.

\begin{lemma} \label{lemma:above set iso} Let $\mathcal{A}=(A;\theta)$ be a monounary algebra and let $x,y\in A$. Suppose  that  for each $n\in \omega$ and each $a\in \theta^{-n}(x)\cap A_x$ and $b\in \theta^{-n}(y)\cap A_y$  we have $|\theta^{-1}(a)|=|\theta^{-1}(b)|$.  Then $\mathcal{A}_x\cong \mathcal{A}_y$. 
\end{lemma} 

\begin{proof} For each $z\in A$ and $k\in \omega$, let $A_z^{(k)} := \bigcup_{0\leq i \leq k}\theta^{-i}(z) \cap A_z$. 

 Let $\psi_0$ be  the unique map between $\{x\}$ and $\{y\}$. 
 Suppose there exists $n\in \omega$ and maps $\psi_0,\psi_1,\dots,\psi_n$ with each $\psi_k\colon A_x^{(k)} \rightarrow A_y^{(k)}$ extending $\psi_{k-1}$.  
  For each $z\in A_x^{(n)}$ there exists a bijection $\pi_z\colon \theta^{-1}(z)\rightarrow \theta^{-1}(\psi_n(z))$  by our hypothesis.
Then the map $\psi_{n+1} = \psi_n \cup \{\pi_z \colon z\in A_x^{(n)} \}$ is an isomorphism from $A_x^{(n+1)}$ to   $A_y^{(n+1)}$. 
We may thus inductively build an isomorphism  $\bigcup_{n\in \omega} \psi_n \colon \mathcal{A}_x \rightarrow \mathcal{A}_y$ as required. 
\end{proof}

\begin{proposition}\label{prop:1hom} Let $\mathcal{A}=(A;\theta)$ be a connected monounary algebra.
 Then $\mathcal{A}$ is  ultrahomogeneous if and only if for all $x,y\in A$, if $\htt(x)=\htt(y)$ then $|\theta^{-1}(x)|=|\theta^{-1}(y)|$. 
\end{proposition} 

\begin{proof} Note that as $\mathcal{A}$ is connected it  follows from Lemma \ref{lemma:cyclic stuff} that $\htt(x) =\htt(y)$ if and only if $\langle x \rangle\cong \langle y \rangle$. 

$(\Rightarrow)$ Let $x,y\in A$ be such that  $\htt(x) =\htt(y)$.
 Then by the  ultrahomogeneity of $\mathcal{A}$ we may extend the isomorphism between $\langle x \rangle$ to $\langle y \rangle$ to an automorphism $\phi$ of $\mathcal{A}$.
 Since $\phi(x)=y$ it follows that $|\theta^{-1}(x)|=|\theta^{-1}(y)|$. 

$(\Leftarrow)$ By Proposition \ref{prop:early hom iff 1hom} it suffices to show that $\mathcal{A}$ is 1-ultrahomogeneous.
 If  $\cyc(\mathcal{A})=\emptyset$ then  $\htt(x)=\infty$ for each $x\in A$, and hence by our hypothesis there exists a cardinal $\alpha$ with $|\theta^{-1}(x)|=\alpha$, so that $\mathcal{A}\cong \mathcal{B}_{\alpha}$  is transitive. 
 Hence, if $\phi\colon \langle x \rangle \rightarrow \langle y \rangle$ is an isomorphism for $x,y\in A$, then there exists an automorphism of $\mathcal{A}$ mapping $x$ to $y$, which clearly extends $\phi$. 
 
Now suppose  $\cyc(\mathcal{A})$ is non-empty.
 We claim that $\mathcal{A}_u\cong \mathcal{A}_v$ if  $\htt(u)=\htt(v)$.  
Indeed, for each $n\in \mathbb{N}$ and for each $a\in \theta^{-n}(u)\cap A_u$ and each $b\in \theta^{-n}(v)\cap A_v$ we have that  $\htt(a) =\htt(b) =\htt(u)+n$, and so $|\theta^{-1}(a)|=|\theta^{-1}(b)|$ by our hypothesis.
 The claim then follows  by Lemma \ref{lemma:above set iso}.
  Now let  $x,y\in A$ be such that $\langle x \rangle\cong \langle y \rangle$, so that $x$ and $y$ have common height $k$, say. 
  Letting $c=\theta^k(x)$ and $d=\theta^k(y)$, there exists by Lemma \ref{lem:cyclichom}  an automorphism $\pi$ of  $\cyc(\mathcal{A})$ which maps $c$ to $d$.
  By our claim we may fix some isomorphism $\phi_u\colon \mathcal{A}_u\rightarrow \mathcal{A}_{\pi(u)}$ each $u\in \cyc(\mathcal{A})$.
   Then $\phi=\bigcup_{u\in \cyc(\mathcal{A})}\phi_u$ is an automorphism of $\mathcal{A}$  by Lemma  \ref{lemma:cycles aut} which maps $c$ to $d$. 
  Since $\mathcal{A}_{\theta^m(x)}\cong \mathcal{A}_{\theta^m(y)}$ for each $m\in \omega$ by our claim,  it follows from Corollary \ref{cor:heightauto} that there exists an automorphism of $\mathcal{A}$ mapping $x$ to $y$.
\end{proof}

We arrive at the following implicit description of all ultrahomogeneous monounary algebras: 

\begin{theorem}\label{thm:mainhom} Let $\mathcal{A}=(A;\theta)$ be a monounary algebra. Then  $\mathcal{A}$ is ultrahomogeneous if and only if whenever $x,y\in A$ are such that $\langle x \rangle \cong \langle y \rangle$, then $|\theta^{-1}(x)|=|\theta^{-1}(y)|$ and the connected components of $\mathcal{A}$ containing $x$ and $y$  are isomorphic.  
 \end{theorem} 
 
 \begin{proof}
$(\Rightarrow)$ If $\langle x \rangle \cong \langle y \rangle$ then there exists an automorphism of $\mathcal{A}$ mapping $x$ to $y$ by ultrahomogeneity. Clearly $|\theta^{-1}(x)|=|\theta^{-1}(y)|$ and the connected components of containing $x$ and $y$ are isomorphic by  Lemma \ref{lemma: iso con}. 

$(\Leftarrow)$   Let $\mathcal{B}\in \mathfrak{C}(\mathcal{A})$. 
Then for any $x,y\in B$, if  $\htt(x) =\htt(y)$ then $\langle x \rangle \cong \langle y \rangle$, and so $|\theta^{-1}(x)|=|\theta^{-1}(y)|$ by our hypothesis.
 Hence $\mathcal{B}$ is   ultrahomogeneous by Proposition \ref{prop:1hom}.
   Now suppose $\mathcal{B},\mathcal{B}'\in \mathfrak{C}(\mathcal{A})$ are such that $|\cyc(\mathcal{B})|=|\cyc(\mathcal{B}')|$. Then for any $c\in\cyc(\mathcal{B})$ and $d\in\cyc(\mathcal{B}')$ we have 
\[ \langle c \rangle = \cyc(\mathcal{B}) \cong \cyc(\mathcal{B}')=\langle d \rangle, 
\] 
and so by our hypothesis  $\mathcal{B}\cong \mathcal{B}'$. Hence $\mathcal{A}$ is ultrahomogeneous by Proposition \ref{prop:connected}.
 \end{proof} 

For an explicit description of all ultrahomogeneous monounary algebras the remaining task is to classify  connected monounary algebras  $\mathcal{A}=(A;\theta)$ satisfying
\begin{equation} \label{eq: ht} \htt(x)=\htt(y) \Rightarrow  |\theta^{-1}(x)|=|\theta^{-1}(y)|. 
\end{equation}
We first consider the case where $\mathcal{A}$ contains no cyclic elements.

\begin{corollary}\label{cor:ultra-hom} Let $\mathcal{A}=(A;\theta)$ be a cycle-free monounary algebra.  Then the following are equivalent: 
\begin{itemize}
\item[(1)]  $\mathcal{A}$ is ultrahomogeneous; 
\item[(2)]  $\mathcal{A}$ is transitive; 
\item[(3)]  there exists cardinals $\alpha,\beta>0$ such that $\mathcal{A}\cong \beta \cdot \mathcal{B}_\alpha$. 
\end{itemize} 
\end{corollary} 

\begin{proof}
(1) $\Leftrightarrow$ (2) Since $\mathcal{A}$ is cycle-free we have that $\langle x\rangle\cong \mathcal{N}$ for every $x\in A$. Hence $\mathcal{A}$ is 1-ultrahomogeneous if and only if it is transitive, to which the result follows by Proposition \ref{prop:early hom iff 1hom} and Theorem \ref{thm:mainhom}. 

(2) $\Leftrightarrow$ (3) Immediate from Theorem \ref{thm:transitive} (or, alternatively, from \cite[Theorem 2.5]{monohom}). 
\end{proof}

On the other hand, connected monounary algebras with elements of finite height and satisfying \eqref{eq: ht} can be built up from cycles: 

 \begin{notation}  For $n\in \mathbb{N}$ and (possibly finite or empty) countable list $\alpha_0,\alpha_1,\dots$ of non-zero cardinals, we let $\mathcal{A}[n;\alpha_0,\alpha_1,\dots]$ denote the connected monounary algebra with: 
 \begin{itemize}
 \item set of cyclic elements $C$ of size $n$, 
 \item $|\theta^{-1}(c)\setminus C|=\alpha_0$ for each $c\in C$, 
 \item if  $\htt(x)=k> 0$  then $|\theta^{-1}(x)|=\alpha_k$.    
\end{itemize}   
\end{notation} 

For example, the infinite monounary algebra $(\mathbb{N};f)$ with $f(1)=1$ and $f(i)=i-1$ for $i>1$ is isomorphic to $\mathcal{A}[1;1,1,\dots]$, while $\mathcal{Z}_n=\mathcal{A}[n]$. 

Since isomorphisms between monounary alebras preserve height and preimages, we have $$\mathcal{A}[n;\alpha_0,\alpha_1,\dots]\cong \mathcal{A}[m;\beta_0,\beta_1,\dots] \text{ if and only if }n=m \text{ and } \alpha_k=\beta_k \text{ for all } k.$$  We have thus shown: 

\begin{theorem}\label{thm:hom con class} A connected monounary algebra   is ultrahomogeneous if and only if it is isomorphic to either $\mathcal{B}_\alpha$ for some cardinal $\alpha >0$, or to $\mathcal{A}[n;\alpha_0,\alpha_1,\dots]$ for some  $n\in \mathbb{N}$  and cardinals $\alpha_0,\alpha_1,\dots >0$. 
\end{theorem} 

Recalling Notation \ref{not:sum}, all ultrahomogeneous monounary algebras can then be explicitly constructed  from the Theorem \ref{thm:hom con class} and Proposition \ref{prop:connected}: 

\begin{theorem}\label{thm:hom class}    A   monounary algebra   is ultrahomogeneous if and only if it is  isomorphic to  
\[  \beta\cdot \mathcal{B}_{\alpha} +  \sum_{n\in \mathbb{N}} \beta_n\cdot \mathcal{A}[n;\alpha^n_0,\alpha^n_1,\dots]
\]   
for some    cardinals $\beta,\beta_n\geq 0$ and $\alpha,\alpha_0^n,\alpha_1^n,\dots >0$.    
\end{theorem}

\begin{corollary}\label{cor:unct homog} There exists $2^\omega$ countable connected ultrahomogeneous monounary algebras, up to isomorphism.  
\end{corollary}  

\begin{proof}
For each $I=\{k_1,k_2,\dots\} \subseteq \mathbb{N}$, consider the connected monounary algebra $\mathcal{A}_I=\mathcal{A}[1;k_1,k_2,\dots]$, so that $\mathcal{A}_I\not\cong \mathcal{A}_J$ whenever $I\neq J$. 
\end{proof}

\begin{remark} Recall   that the classes $\mathfrak{F}$ and $\mathfrak{F}_k$ $(k\geq 1)$, defined at the start of the section, were shown to be  Fra\"iss\'e classes  in \text{\cite{Sokic}}. By Theorem \ref{thm:hom class} we may determine their Fra\"iss\'e limits $\mathbb{F}$ and $\mathbb{F}_k$, respectively, as follows:
\begin{align*}
& \mathbb{F}\cong    \sum_{n\in \mathbb{N}} \omega \cdot \mathcal{A}[n;\omega, \omega,\dots], \quad \mathbb{F}_k \cong  \sum_{n\in \mathbb{N}} \omega \cdot \mathcal{A}[n;k-1,k,k,\dots].
\end{align*}
 \end{remark} 
 
 \section{Links to related concepts} 

In this section we consider a number of other properties  on the automorphism group of a monounary algebra. We  show that another well-known  property, after homomorphism-homogeity, possesses a mirroring of Proposition \ref{prop:1hom}. We then use Theorem \ref{thm:hom class}  to provide several classifications of properties of a similar vein to ultrahomogeneity, namely homogeneity, partial-homogeneity, and the ultrahomogeneity of the relational form of a monounary algebra. 

\subsection{$\omega$-categoricity} 

A countable structure $\mathcal{M}$ is called $\omega$-categorical if it can be defined uniquely, up to isomorphism, by its first order properties.
 The much celebrated Ryll-Nardzewski Theorem states that $\omega$-categoricity is equivalent to the purely algebraic condition  that the automorphism group $\Aut(\mathcal{M})$  is oligomorphic. Note that $\omega$-categorical structures are necessarily {locally finite}, that is,   every finitely generated substructure is finite. Locally finite monounary algebras are simple to describe: 
 
 \begin{lemma}\label{lemma:LF}  Let $\mathcal{A} = (A; \theta)$ be a monounary algebra. Then the following conditions are
equivalent:
\begin{itemize}
\item[$(1)$] $\mathcal{A}$  is locally finite;
\item[$(2)$]  every element of $\mathcal{A}$ has finite height;
\item[$(3)$]  every 1-generated subalgebra of $\mathcal{A}$ is finite;
\item[$(4)$]  every connected component of $\mathcal{A}$ has a cycle.
\end{itemize} 
\end{lemma} 

\begin{proof} Let $\mathcal{A}$ be a monounary algebra with finite subset $B$.
 Then $\langle B\rangle = \bigcup_{b\in B} \langle b \rangle$ is finite if and only if $\langle b \rangle$ is finite for each $b\in B$, if and only if  $\htt(b)$ is finite for each $b\in B$. By Lemma \ref{lemma: subalgebra} this occurs if and only if  each connected component of $\mathcal{A}$ which intersects $B$ non-trivially has a cycle, and the result follows. 
\end{proof}

\begin{corollary} \label{cor:HomogLF} A locally finite monounary algebra is ultrahomogeneous if and only if it is isomorphic to $\sum_{n\in \mathbb{N}} \beta_n\cdot \mathcal{A}[n;\alpha^n_0,\alpha^n_1,\dots]$
for some    cardinals $\beta_n$ and $\alpha_0^n,\alpha_1^n,\dots >0. $ 
\end{corollary} 

\begin{proof} Since $\mathcal{B}_\alpha$ is cycle-free  for any cardinal $\alpha$, it is not locally finite. On the other hand,   $\mathcal{A}[n;\alpha_0,\alpha_1,\dots]$  is locally finite for any $n\in \mathbb{N}$ and cardinals $\alpha_0,\alpha_1,\dots$, to which the result follows. 
\end{proof}

By Fra\"iss\'e's Theorem and the corollary above we may therefore answer  the first half of Problem 1 of \cite{Sokic}: 

\begin{corollary} A class of finite monounary algebras is a Fra\"iss\'e  class if and only if it is the age of some monounary algebra isomorphic to 
$ \sum_{n\in \mathbb{N}} \beta_n\cdot \mathcal{A}[n;\alpha^n_0,\alpha^n_1,\dots] 
$
for some  $\beta_n\in \omega \cup \{\omega\}$  and $ \alpha_0^n,\alpha_1^n,\dots \in \mathbb{N}\cup \{\omega\} $. 
\end{corollary}

Recall that $\omega$-categorical monounary algebras were classified by Shishmarev  in \cite{catmono}. A monounary algebra $\mathcal{A}$ is called \textit{limited} if there exists $N\in \mathbb{N}$ such that for all $a\in A$ there exists $n,m\leq N$ with $\theta^{n}(a)=\theta^{n+m}(a)$. 

\begin{theorem}[\protect{\cite[Theorem 1]{catmono}}] \label{thm:catmononew}  A  monounary algebra $\mathcal{A}$ is $\omega$-categorical if and only if  the following conditions hold: 
\begin{itemize}
\item[$(\mathrm{i})$]  $\mathcal{A}$  is limited; 
\item[$(\mathrm{ii})$] there are only finitely many partial subalgebra of the form $ \bigcup_{k\in \omega} \theta^{-k}(x)$ $(x\in A)$, up to isomorphism. 
\end{itemize} 
\end{theorem}

Our chief result of this section is showing that,  for locally finite  monounary algebras,  $\omega$-categoricity is equivalent to its   pointwise property, thus mirroring Proposition \ref{prop:early hom iff 1hom}:

\begin{theorem} A monounary algebra is $\omega$-categorical if and only if it is locally finite and has finitely many 1-orbits. 
\end{theorem} 

\begin{proof}
$(\Leftarrow)$ Suppose $\mathcal{A}$ is locally finite and has finitely many 1-orbits. 
By local finiteness each element has finite height. Since height is preserved by automorphisms, and as  $o_1(\mathcal{A})$ is finite, it follows that there can exist only finitely many different heights. 
Hence, there exists $n\in \mathbb{N}$ such that $\theta^n(x)\in\cyc(\mathcal{A})$ for all $x\in A$. Similarly, as cycles are 1-generated, there exists a bound $m\in \mathbb{N}$ on their cardinality. Taking $N=$ max$\{m,n\}$ shows that $\mathcal{A}$ is limited.  
Since automorphisms preserve preimages,  the number of non-isomorphic partial subalgebra of the form $ \bigcup_{k\in \omega} \theta^{-k}(x)$  is at most  $o_1(\mathcal{A})$. Hence $\mathcal{A}$ is $\omega$-categorical by Theorem \ref{thm:catmononew}.

$(\Rightarrow$) Immediate.  
\end{proof}

\begin{corollary} \label{cor:homog + cat} Let $\mathcal{A}$ be a countable monounary algebra. Then the following are equivalent: 
\begin{itemize}
\item[$(1)$] $\mathcal{A}$ is $\omega$-categorical and ultrahomogeneous; 
\item[$(2)$] $\mathcal{A}$ is limited and ultrahomogeneous; 
\item[$(3)$] $\mathcal{A}$  is isomorphic to 
$\sum_{1\leq n \leq k} \beta_n\cdot \mathcal{A}[n;\alpha^n_0,\alpha^n_1,\dots,\alpha^n_{h_n}] 
$
for some $k\in \mathbb{N}$, $h_1,\dots,h_n\in \omega$,  $\beta_n\in \omega\cup \{\omega\}$,   and  $ \alpha_0^n,\dots,\alpha_{h_n}^n\in \mathbb{N} \cup \{\omega\}$. 
\end{itemize}
\end{corollary} 

\begin{proof}
$(1) \Rightarrow (2)$ Immediate from Theorem \ref{thm:catmononew}. 

$(2) \Rightarrow (3)$ Let $N\in \mathbb{N}$ witness that $\mathcal{A}$ is limited. Then   $\htt(\mathcal{A})\leq N$ so that $\mathcal{A}$ is locally finite, and hence  is isomorphic to some $\sum_{n\in \mathbb{N}} \beta_n\cdot \mathcal{A}[n;\alpha^n_0,\alpha^n_1,\dots,\alpha^n_{h_n}]$ by Corollary \ref{cor:HomogLF}, where $h_n+1 \leq N$ for each $n\in \mathbb{N}$.  Moreover,  the cycles of $\mathcal{A}$ are of size at most $N$, so that $\beta_k=0$ for all $k>N$. 

$(3) \Rightarrow (1)$ Since $\htt(\mathcal{A}) \leq $ max$\{h_n + 1 \colon 1\leq n \leq k\}$  is finite and $\mathcal{A}$ has only finitely many non-isomorphic cycles, it is limited. Moreover, the number of non-isomorphic partial subalgebra of the form $\bigcup_{k\in \omega} \theta^{-k}(x)$ ($x\in A)$  is bound from above by $\sum_{n\leq N} (h_n+1)$. 
The result then follows by Theorem \ref{thm:catmononew}.
\end{proof}

Alternatively, Corollary \ref{cor:homog + cat} provides a description of those monounary algebras which are both  $\omega$-categorical and have \textit{quantifier elimination} (see \cite[Corollary 6.4.2]{Hodges97}). A classification of all quantifier eliminable monounary algebras can be found in \cite{QEmono}. 

 \subsection{Homogeneity and partial homogeneity}

Recall that for  $n\in \mathbb{N}$, a monounary algebra $\mathcal{A}$ is called  \textit{(partially) $n$-homogeneous} if every isomorphism between (partial) subalgebras of size $n$ extends to an automorphism of $\mathcal{A}$; if  $\mathcal{A}$ is (partially) $n$-homogeneous for each $n\in \mathbb{N}$ then it is called \textit{(partially) homogeneous}.

\subsubsection{Homogeneity}
The 1- and 2-homogeneity of monounary algebras were described in \cite{1hom, 2-hom}.  

Note that a monounary algebra $\mathcal{A}$ can be written as $\mathcal{A}_1 + \mathcal{A}_2$, where $\mathcal{A}_1$ is the maximal cycle-free subalgebra of $\mathcal{A}$ consisting of  elements of infinite height, so that $\mathcal{A}_2$  consists of those elements of finite height. 

\begin{theorem}
     Let $\mathcal{A}=(A;\theta)$ be a monounary algebra, say $\mathcal{A} = \mathcal{A}_1+\mathcal{A}_2$ where $\mathcal{A}_1$  is the maximal
cycle-free subalgebra of $\mathcal{A}$.
 Then the following conditions are equivalent: 
 \begin{itemize}
 \item[$(1)$] $\mathcal{A}$ is homogeneous;
 \item[$(2)$] $\mathcal{A}_2$ is ultrahomogeneous; 
\item[$(3)$]  whenever $x, y \in A_2$ are such that $\langle x \rangle \cong \langle y \rangle$, then $|\theta^{-1}(x)|= |\theta^{-1}(y)|$ 
and the connected components of  $\mathcal{A}$ containing $x$ and $y$ are isomorphic; 
\item[$(4)$]  $\mathcal{A}_2\cong  \sum_{n\in \mathbb{N}} \beta_n\cdot \mathcal{A}[n;\alpha^n_0,\alpha^n_1,\dots]$
for some    cardinals $\beta_n\geq 0$ and $\alpha_0^n,\alpha_1^n,\dots > 0. $ 
 \end{itemize}
 
\end{theorem}

\begin{proof}
$(1) \Leftrightarrow (2)$   Since a finite subalgebra intersects trivially with $\mathcal{A}_1$, and since any automorphism of $\mathcal{A}_2$ extends to an automorphism of $\mathcal{A}$ (by pointwise fixing $\mathcal{A}_1$),  it follows that $\mathcal{A}$ is homogeneous if and only if $\mathcal{A}_2$ is. Since $\mathcal{A}_2$ is locally finite by Lemma \ref{lemma:LF}, it is homogeneous if and only if it is ultrahomogeneous. 

$(2) \Leftrightarrow (3)$ Theorem \ref{thm:mainhom}. 

$(2) \Leftrightarrow (4)$ Since $\mathcal{A}_2$ is locally finite, the result is immediate by Corollary \ref{cor:HomogLF}.
\end{proof}

\subsubsection{Partial homogeneity}

The partial 1- and 2-homogeneity of monounary algebras were described in \cite{1hom, partial2hom}.  
  Note that substructures of the relational form $\mathcal{A}^*$  of $\mathcal{A}$ correspond to partial subalgebras of $\mathcal{A}$, and moreover the isomorphisms between finite substructures of $\mathcal{A}^*$ are the same as the isomorphisms between finite partial subalgebras of $\mathcal{A}$. It follows that $\mathcal{A}^*$ is $n$-homogeneous if and only if $\mathcal{A}$ is partially $n$-homogeneous.
 
 \begin{theorem}\label{thm:hom relationalform}
 Let $\mathcal{A}$ be a monounary algebra. Then the following are equivalent: 
 \begin{enumerate} 
  \item[$(1)$] $\mathcal{A}$ is  partially homogeneous;
  \item[$(2)$] $\mathcal{A}$ is partially 1- and 2-homogeneous;
 \item[$(3)$]  $\mathcal{A}^*$ is ultrahomogeneous; 
 \item[$(4)$] $\mathcal{A}^*$ is 2-ultrahomogeneous;
 \item[$(5)$]     $\mathcal{A}$ is isomorphic to either
 \begin{itemize}
 \item[$(\mathrm{i})$] $\alpha \cdot \mathcal{Z}_1 + \beta \cdot \mathcal{Z}_2$  for some cardinals $\alpha,\beta$;   
 \item[$(\mathrm{ii})$] $\alpha \cdot \mathcal{Z}_1 + \beta \cdot \mathcal{Z}_3$ for some cardinals $\alpha,\beta$;  
 \item[$(\mathrm{iii})$] $\alpha \cdot \mathcal{Z}_1 + \mathcal{Z}_4$  for some cardinal $\alpha$;  
 \item[$(\mathrm{iv})$] $\alpha \cdot \mathcal{A}[1;1]$  for some cardinal $\alpha$; 
 \item[$(\mathrm{v})$] $\mathcal{A}[1;\alpha]$  for some cardinal $\alpha$.  
 \end{itemize}
\end{enumerate}
In particular, if $\mathcal{A}^*$ is ultrahomogeneous then $\mathcal{A}$ is ultrahomogeneous. 
 \end{theorem}
 
 \begin{proof}
 $(1) \Rightarrow (3)$ Since $\mathcal{A}^*$ is a relational structure, and thus locally finite, it is ultrahomogeneous if and only if it is $n$-homogeneous for each $n$. The result then follows by the note above. 
 
 (3) $\Rightarrow$ (4) Immediate. 

 (4) $\Rightarrow$ (2) The 2-generated substructures of $\mathcal{A}^*$ correspond to partial subalgebras of $\mathcal{A}$ of size 1 or 2, and so the result again follows by the note above. 
 
 (2) $\Rightarrow$ (5)  
 By \cite[Theorem 4.8]{partial2hom}   $\mathcal{A}$ is partially 2-homogeneous if and only if it is isomorphic to   either one of (i)-(v)  or $\mathcal{D}_\alpha=\mathcal{A}[1;1,\alpha]$ for some cardinal $\alpha>0$. 
However, $\mathcal{D}_\alpha$ is not partially 1-homogeneous by \cite[Lemma 2.4]{1hom}. 

(5) $\Rightarrow$ (1)  
Note that each case is ultrahomogeneous by Theorem \ref{thm:hom class}, so it suffices to show that isomorphic finite partial subalgebras which are not subalgebra extend (in $\mathcal{A}$) to isomorphic finitely generated subalgebra of $\mathcal{A}$.
This is a simple exercise for cases (i)-(iii). 
In case (v), any partial subalgebra which  is not a subalgebra has a partial unary operation with empty domain. Hence its relational form is a null digraph  (a digraph with no edges). Such partial subalgebra are isomorphic if and only if they are of the same size; by adding the 1-cycle we obtain isomorphic subalgebra. Case (iv) is proven similarly. 
 \end{proof}
 
Note that, unlike in Proposition \ref{prop:early hom iff 1hom}, ultrahomogeneity is not equivalent to 1-ultrahomogeneity for the relational form of monounary algebras. For example, $\mathcal{Z}_n^*$ is transitive  and hence 1-ultrahomogeneous for any $n\in \mathbb{N}$, but for $n>4$ it is not 2-ultrahomogeneous.  
 
Recall that there are uncountably many countable ultrahomogeneous monounary algebras by Corollary \ref{cor:unct homog}. 
On the other hand, as a consequence to Theorem \ref{thm:hom relationalform}   there exists only countably many countable ultrahomogeneous relational forms of monounary algebras.

\subsection{The poset of  variants of homogeneity }

To summarise our work on various concepts of homogeneity of monounary algebras, we use the following notation: 
\begin{itemize}
    \item $\mathcal{UH}$ ($\mathcal{UH}_n$) - the class of all ($n$-)ultrahomogeneous monounary algebras 
    \item $\mathcal{H}$ ($\mathcal{H}_n$) - the class of all ($n$-)homogeneous monounary algebras 
    \item $\mathcal{PH}$ ($\mathcal{PH}_n$) - the class of all partially ($n$-)homogeneous monounary algebras 
    \item $\mathcal{T}$ - the class of all transitive monounary algebras. 
\end{itemize}

\begin{theorem} \label{thm: lattice of conditions} We have 
\begin{center}
\begin{tikzpicture}
 \node (a) at (-3,0) {$(\dagger)$}; 
  \node (max) at (0,0) {$\mathcal{PH} = \mathcal{PH}_1 \cap \mathcal{PH}_2$};
   \node (u) at (1.45,0.25) {$\udsubseteq$};
      \node (v) at (1.45,-0.25) {$\ddsubseteq$};
 \node (g) at (1.95,-0.5) {$\mathcal{PH}_1$};
\node (h) at (1.95,0.5) {$\mathcal{PH}_2$};
  \node (w) at (2.4,0.25) {$\ddsubseteq$};
      \node (x) at (2.4,-0.25) {$\udsubseteq$};
 \node (i) at (4.65,0) {$\mathcal{UH} = \mathcal{UH}_1  \subsetneq \mathcal{H} \subsetneq \mathcal{H}_2 \subsetneq \mathcal{H}_1$. };
\end{tikzpicture}
\end{center}
Moreover, 
\begin{itemize}
    \item[ $(\mathrm{i})$] $\mathcal{H}_{n-1} \nsubseteq \mathcal{H}_n$ and $\mathcal{H}_{n+1} \nsubseteq \mathcal{H}_{n}$ for each $n>1$. 
    \item[$(\mathrm{ii})$] $\mathcal{PH}_2\subsetneq \mathcal{PH}_n$ and $\mathcal{PH}_{n} \nsubseteq \mathcal{UH}$ for each $n>2$. 
    \item[$(\mathrm{iii})$]  $\mathcal{PH}_1\perp \mathcal{PH}_n$  for each $n>1$. 
      \item[$(\mathrm{iv})$]  $\mathcal{T}\subsetneq\mathcal{PH}_1$ and $\mathcal{T}\perp \mathcal{PH}_{n}$  for each $n>1$.
\end{itemize}  
\end{theorem}

\begin{proof} For each $n\in \mathbb{N}$, we   let $\mathcal{Z}_n^1$ denote the monounary algebra $\mathcal{Z}_n \cup \{x\}$ where $x$ is the unique acyclic element, with $\theta(x) =0$. 

(i) $\mathcal{H}_{n-1} \nsubseteq \mathcal{H}_n$ $(n>1)$: 
 We have $\mathcal{Z}_n^1\notin \mathcal{H}_n$, while as $\mathcal{Z}_n^1$ contains no subalgebra of size $n-1$ we trivially have   $\mathcal{Z}_n^1\in \mathcal{H}_{n-1}$. 

$\mathcal{H}_{n+1} \nsubseteq \mathcal{H}_n$ $(n>1)$: We have $\mathcal{Z}_n^1 \in \mathcal{H}_{n+1} \setminus \mathcal{H}_n$. 

(ii)  $\mathcal{PH}_2\subsetneq \mathcal{PH}_n$ $(n>2)$: Recall that a partially 2-homogeneous monounary algebra which is not partially homogeneous is isomorphic to some $\mathcal{A}[1;1,\alpha]$. It is a simple exercise to show that this is partially $n$-homogeneous. Moreover, 
 $\mathcal{Z}_{n-1}^1\in \mathcal{PH}_n \setminus \mathcal{PH}_2$.

 $\mathcal{PH}_{n} \nsubseteq \mathcal{UH}$ $(n>2)$: Similarly, $\mathcal{Z}_{n-1}^1\in \mathcal{PH}_n \setminus \mathcal{UH}$.     

(iii) $\mathcal{PH}_1\perp \mathcal{PH}_n$ ($n>1$): We have $\mathcal{Z}_{2n+1}\in \mathcal{PH}_1  \setminus \mathcal{PH}_n$  since it contains an $n$-element antichain, and $\mathcal{A}[1;1,\alpha]\in \mathcal{PH}_n \setminus \mathcal{PH}_1$. 

(iv) $\mathcal{T}\subseteq  \mathcal{PH}_1$: Immediate. 

$\mathcal{T}\perp \mathcal{PH}_{n}$ $(n>1)$: By Theorem \ref{thm:transitive} and Theorem \ref{thm:hom relationalform}  we have $\mathcal{Z}_{2n+1}\in \mathcal{T}\setminus \mathcal{PH}_n$ and  $\mathcal{Z}_1 + \mathcal{Z}_2\in \mathcal{PH}\setminus \mathcal{T}$. This also shows that $\mathcal{T}\neq  \mathcal{PH}_1$. 

$\mathcal{PH} = \mathcal{PH}_1 \cap \mathcal{PH}_2 \subsetneq \mathcal{PH}_2$: Follows from Theorem \ref{thm:hom relationalform} and its proof. 

$\mathcal{PH}  \subsetneq \mathcal{PH}_1$: Inclusion is immediate, while $\mathcal{Z}_5\in \mathcal{PH}_1 \setminus \mathcal{PH}$. 

$\mathcal{PH}_n  \subsetneq \mathcal{UH}$ ($n=1,2$): Inclusion follows from \cite[Theorem 2.8]{1hom}, \cite[Theorem 4.8]{partial2hom}, and Theorem \ref{thm:hom class}. Moreover, $\mathcal{Z}_2 + \mathcal{Z}_3 \in \mathcal{UH} \setminus (\mathcal{PH}_1 \cup \mathcal{PH}_2)$. 

$\mathcal{UH}=\mathcal{UH}_1 \subsetneq \mathcal{H}$: The first equality is from Theorem \ref{thm:mainhom}. Any cycle-free monounary algebra is automomatically in $\mathcal{H}$, but not necessarily  in $\mathcal{UH}$, for example $(\mathbb{N};\text{Suc})$.  

$\mathcal{H} \subsetneq \mathcal{H}_2$: Inclusion is immediate, while $\mathcal{Z}_3^1\in \mathcal{H}_2 \setminus \mathcal{H}$. 

  $\mathcal{H}_2 \subsetneq \mathcal{H}_1$: Inclusion follows from  \cite[Theorem 5.6]{2-hom} and \cite[Theorem 5.6]{1hom}.  Moreover, $\mathcal{Z}_2^1\in \mathcal{H}_1 \setminus \mathcal{H}_2$. 
\end{proof}

Note that each of the classes in $(\dagger)$ have now been described, while a full description of $\mathcal{H}_n$ and $\mathcal{PH}_n$ for $n>2$ is outside of the scope of this paper.

\subsection{Ultrahomogeneous directed pseudoforests}

Recall that a digraph is a directed pseudoforest if and only if it is the relational form  of a partial monounary algebra  without 1-cycles. We end this section by giving a simple classification of ultrahomogeneous directed pseudoforests in arbitrary cardinalities; a description of all homogeneous countable digraphs is given by G. Cherlin   in \cite{CherlinDigraph}. 

\begin{lemma}\label{lemma:partial hom} Let $\mathcal{A}=(A;\theta)$ be a partial monounary algebra in which dom$(\theta)\neq A$. Then $\mathcal{A}^*$ is ultrahomogeneous if and only if dom$(\theta)=\{x:\theta(x)=x\}$. 
\end{lemma} 
 
 \begin{proof} ($\Rightarrow$) 
 If $x,y\in A$ with $x\notin$dom($\theta$) and $\theta(y)\neq y$, then $\{x\}$ and $\{y\}$ are isomorphic substructures of $\mathcal{A}^*$, but no automorphism can map $x$ to $y$. 
 
 ($\Leftarrow$) Since $\mathcal{A}^*$ is the union of loops and a null digraph,  it is clearly ultrahomogeneous. 
 \end{proof}
 
Our desired classification then follows immediately from  Theorem \ref{thm:hom relationalform} and Lemma \ref{lemma:partial hom}. 
 
 \begin{theorem} Let $\mathcal{M}$ be a directed pseudoforest. Then $\mathcal{M}$ is ultrahomogeneous if and only if $\mathcal{M}=\mathcal{A}^*$ where $\mathcal{A}$ is a partial monounary algebra  isomorphic to either 
 \begin{enumerate}
 \item[$(\mathrm{i})$] $\alpha\cdot \mathcal{Z}_2$ for some cardinal $\alpha$; 
  \item[$(\mathrm{ii})$] $\alpha \cdot \mathcal{Z}_3$ for some cardinal $\alpha$; 
   \item[$(\mathrm{iii})$] $ \mathcal{Z}_4$; 
   \item[$(\mathrm{iv})$] $(A;\theta)$ with dom$(\theta)=\emptyset$. 
\end{enumerate}  
 \end{theorem}

\section{Towards unary algebras}

Given the success in determining the ultrahomogeneity  of monounary algebras, it seems reasonable to attempt  to extend to unary algebras.  

A possible first step is to determine the connections between the ultrahomogeneity of  a unary algebra $\mathcal{A}=(A;\theta_1,\dots,\theta_n)$ and  the ultrahomogeneity of its  monounary reducts $(A;\theta_i)$. 
Unfortunately,  the monounary algebra reducts of a unary algebra need not inherit ultrahomogeneity: 

\begin{example} Let $A=\{x,y,z\}$ and $\theta_1$ and $\theta_2$ be unary maps defined by $\theta_1(x)=\theta_1(y)=z$, $\theta_1(z)=y$, $\theta_2(z)=\theta_2(y)=x$, and $\theta_2(x)=y$.
 Then $(A;\theta_1)\cong (A;\theta_2)\cong \mathcal{Z}_2^1$ are a 2-cycle with a single acyclic element, and are thus not ultrahomogeneous by Theorem \ref{thm:hom con class}.
However $(A;\theta_1,\theta_2)$ has no proper subalgebras, and  is thus trivially ultrahomogeneous. 
\end{example} 

Moreover,  the example below shows that having ultrahomogeneous monounary reducts is not sufficient for ultrahomogeneity, even  for \textit{commutative unary algebras}, i.e.,  those in which $\theta_i\theta_j=\theta_j\theta_i$ for each $i,j\leq n$. 

\begin{example}\label{ex} Let $A_0=\{c,a_i \colon i \in \mathbb{N}\}$ and $A=A_0\cup B$ for some countably infinite set $B$ disjoint to $A_0$. We define a pair of unary functions $\theta_1$ and $\theta_2$ on $A$ such that each $(A;\theta_i)$ is isomorphic to the ultrahomogeneous monounary algebra $\mathcal{A}[1;\omega,\omega]$ as follows. 
Let $\{X_i \colon i\in \mathbb{N}\}$ and  $\{Y_i \colon i\in \mathbb{N}\}$ be a pair of partitions of $B$ such that $|X_i|=|Y_i|=\aleph_0$ and $|X_i \cap Y_i|=i$ for each $i\in \mathbb{N}$. 
 Let $\theta_k(A_0)=c$ $(k=1,2)$ and $\theta_1(X_i)=\theta_2(Y_i) = a_i$  for each $i\in \mathbb{N}$.
 Let $\mathcal{A}=(A;\theta_1,\theta_2)$, noting that $\mathcal{A}$ is commutative as $$\theta_1\theta_2(A)=\theta_1(A_0) = c = \theta_2(A_0)=\theta_2\theta_1(A).$$
Note that  $\langle a_i \rangle = \{a_i,c\} \cong \{a_j,c\}=\langle a_j \rangle$ for each $i,j\in \mathbb{N}$.
Suppose there exists $\phi\in \Aut(\mathcal{A})$ which maps $a_i$ to $a_j$ (and hence extends the unique isomorphism between $\langle a_i \rangle$ and $\langle a_j  \rangle$).
   Then as $\phi$ preserves $\theta_1$ we have 
   $$\phi(X_i)=\phi(\theta_1^{-1}(a_i))=\theta_1^{-1}(\phi(a_i))=\theta_1^{-1}(a_j)=X_j,$$  and similarly as $\phi$ preserves $\theta_2$ we have $\phi(Y_i)=Y_j$. Hence as $\phi$ is a bijection we have $\phi(X_i\cap Y_i)= X_j\cap Y_j$, and so their cardinalities force $i=j$. Hence $\mathcal{A}$ is not ultrahomogeneous. 
\end{example}

It is an open problem whether Proposition \ref{prop:early hom iff 1hom} can be generalised to unary algebras,
in particular we ask:

\begin{problem} Is every 1-ultrahomogeneous unary algebra necessarily ultrahomogeneous? 
\end{problem}

We may ask this question of other variants of homogeneity, in particular: 

\begin{problem} Is every 1-homomorphism-homogeneous unary algebra necessarily
 homomorphism-homogeneous? 
\end{problem}

It terms of $\omega$-categoricity, we note that if $\mathcal{A}=(A;\theta_1,\dots,\theta_n)$ is $\omega$-categorical then it follows from oligomorphicity that  each $(A;\theta_i)$ is $\omega$-categorical, since $\Aut(\mathcal{A})\subseteq \Aut(A;\theta_i)$.
  Unfortunately, Example \ref{ex} shows that the converse needs not hold. Indeed, each monounary reduct is isomorphic to $\mathcal{A}[1;\omega,\omega]$ and hence $\omega$-categorical, however we have shown that $\{a_i:i\in \mathbb{N}\}$ is an infinite set of distinct 1-orbits of $\mathcal{A}$.

\begin{problem} Is every countable locally  finite unary algebra with finitely many 1-orbits necessarily  $\omega$-categorical?\footnote{Problems 1 and 3 were recently shown not to hold in general in \cite{unary}.}
\end{problem}

\section*{Acknowledgements} 

Part of this work was funded by the German Science Foundation (DFG, project number 622397)
and by the European Research Council (Grant Agreement no.  681988, CSP-Infinity). 
I would like to thank the anonymous referees whose  insightful suggestions vastly improved the paper.

\end{document}